\documentclass[11pt]{amsart}

\usepackage{amsmath,amsthm,amssymb,mathrsfs,amsfonts,verbatim,enumitem,color,leftidx}
\usepackage{mathabx}
\usepackage{etoolbox} 
\usepackage{bbm}
\usepackage[all,tips]{xy}
\usepackage{graphicx,ifpdf}
\usepackage{stmaryrd}
\usepackage{mathtools}
\usepackage{cancel}
\ifpdf
\fi
\usepackage{hyperref}

\newtheorem{thm}{Theorem}[section]
\newtheorem{lem}[thm]{Lemma}
\newtheorem{cor}[thm]{Corollary}
\newtheorem{prop}[thm]{Proposition}

\theoremstyle{definition}
\newtheorem{de}[thm]{Definition}
\newtheorem{ex}[thm]{Example}

\theoremstyle{remark}

\numberwithin{equation}{section}

\makeatletter

\newcommand{\Rmnum}[1]{\expandafter\@slowromancap\romannumeral #1@}
\makeatother

\newcommand{\ba}{\mathbf{a}}

\newcommand{\be}{\mathbf{e}}

\newcommand{\bp}{\mathbf{p}}
\newcommand{\bq}{\mathbf{q}}

\newcommand{\bs}{\mathbf{\overline{a}}}
\newcommand{\bt}{{\mathbf{a}i-\overline{\mathbf{a}}\ell}}

\newcommand{\bv}{\mathbf{v}}
\newcommand{\bw}{\mathbf{w}}
\newcommand{\bx}{\mathbf{x}}
\newcommand{\by}{\mathbf{y}}

\newcommand{\bK}{\mathbf{K}}
\newcommand{\bZ}{\mathbf{Z}}
\newcommand{\bS}{\mathbf{S}}

\newcommand{\bB}{\mathbf{B}}

\newcommand{\F}{\mathbb{F}}

\newcommand{\N}{\mathbb{N}}

\newcommand{\G}{\Gamma}

\newcommand{\supp}{\mathrm{supp\,}}

\begin{document}

\title[Equidistribution in a local field of positive characteristic]{Equidistribution with an error rate and Diophantine approximation over a local field of positive characteristic }

\author{Sanghoon Kwon}

\address{Center for Mathematical Challenges, Korea Institute For Advanced Study, Seoul 02455, Korea}
\email{skwon@kias.ac.kr}

\author{Seonhee Lim}
\address{Department of Mathematical Sciences, Seoul National University, Seoul 08826, Korea}

\email{slim@snu.ac.kr}
\thanks{}

\thanks{}


\subjclass[2000]{Primary   28A33; Secondary 37A15, 22E40.}

\date{}


\keywords{field of formal series, effective equidistribution, ergodic theorem}

\maketitle
\begin{abstract}
For a local field $\bK$ of formal Laurent series and its ring $\bZ$ of polynomials, we prove a pointwise equidistribution with an error rate of each $H$-orbit in $SL(d,\bK)/SL(d,\bZ)$ for a certain proper subgroup $H$ of horospherical group, extending a work of Kleinbock-Shi-Weiss. 

We obtain an asymptotic formula for the number of integral solutions to the Diophantine inequalities with weights, generalizing a result of Dodson-Kristensen-Levesley.
This result enables us to show pointwise equidistribution for unbounded functions of class $C_\alpha$. 
\end{abstract}

\section{Introduction}
Let $\bK$ be the field $\F_q(\!( t^{-1})\!)$ of formal Laurent series in $t^{-1}$ over a finite field $\F_q$ of order $q$ and let $\bZ$ be the ring $\F_q[t]$ of polynomials in $t$ over $\F_q$. The absolute value $|\cdot|$ on $\bK$ is given by $|f|=q^{\textrm{deg}(f)}$. Let $\mathcal{O}$ be the ring of formal power series $\mathbb{F}_q[\![t^{-1}]\!]$ and $\lambda$ be the Haar measure on $\mathbf{K}$ normalized by $\lambda(\mathcal{O})=1$. 
For $m,n \in \mathbb{N}$ and $d = m+n$, let $G = SL(d,\bK)$ and $\G = SL(d,\bZ)$.  
Let us denote by $\mathfrak{a}^+$ the set of $d$-tuples $\ba=(a_1,\cdots,a_{d})\in\mathbb{N}^{d}$ such that
$$\sum_{i=1}^m a_i=\sum_{j=1}^n a_{m+j}.$$
Given $\ba =(a_1,\cdots,a_{d})\in \mathfrak{a}^+$, let us define
$$g_\ba=\textrm{diag}(t^{a_1},\cdots,t^{a_m},t^{-a_{m+1}},\cdots,t^{-a_{m+n}}).$$ 
In this paper, we study the action of the subsemigroup generated by $g_{\mathbf{a}}$ of $G$ on the space $X= G/\G$. Since $g_\ba$-action on $X$ is ergodic 
with respect to the Haar probability measure $\mu$ on $X$, for $\mu$-almost every $x \in X$, we have $$\lim_{N\to\infty}\frac{1}{N}\sum_{n=0}^{N-1}\varphi(g_\ba^n x)=\int_{X}\varphi d\mu.$$ 

We are interested in the following question: given a proper subgroup $L$ of $G$, does the above convergence of Birkhoff average of $g_\ba$-translates still hold for almost every points in an $L$-orbit? This question was studied in \cite{ksw16} and \cite{shi15} for a real Lie group $G$ and a lattice subgroup $\Gamma$ of $G$.

When $L$ is the horospherical subgroup 
$$H^+=
\{u\in G|g_\ba^{-n}ug_\ba^n\to e\textrm{ as }n\to\infty\}$$ of $G$, then the answer is affirmative (cf. \cite{mo11}). In this article, we consider a subgroup $H$ of $H^+$
\begin{equation}\label{uA} 
H=\left\{u_A=\left(
\begin{array}{cc}
I_m & A \\
0 & I_n
\end{array}\right)\colon\,A\in \textrm{Mat}_{m\times n}(\bK)\right\},
\end{equation}
which has applications in metric Diophantine approximation (Propsition~\ref{prop:1.2} and Theorem~\ref{cwd} below). Since $H \simeq \textrm{Mat}_{m\times n}(\bK)$, the Haar measure on $H$ is isomorphic to $\lambda^{mn}$, which we will denote by $\lambda_H$.
Let us also denote the Haar measure on $G$ such that $\lambda_G(K)=1$ for a maximal compact subgroup $K$ by $\lambda_G$.

We first show that for compactly supported smooth functions, pointwise ergodic theorem for $g_\ba$-action holds for almost every points in each $H$-orbit. A function $f\colon G \to \mathbb{R}$ is \emph{smooth} if there is a compact open subgroup $U \subset G$ such that $f$ is $U$-invariant, i.e. it is locally constant. Let us denote by $C^\infty_c(X)$ the set of functions on $X=G/\G$ whose lifts on $G$ are $\Gamma$-invariant smooth functions on $G$.
\begin{thm}\label{pe}For any given $x\in X$, $\varphi\in C^\infty_c(X)$ and $\epsilon>0,$ as $N \to \infty$, we have
$$\frac{1}{N}\sum_{n=0}^{N-1}\varphi(g_\ba^nhx)=\int_{X}\varphi d\mu+O(N^{-\frac{1}{2}}(\log N)^{\frac{3}{2}+\epsilon}),$$
for $\lambda_H$-almost every $h\in H$.
\end{thm}

To show a similar result for functions with non-compact support, we make use of lattice point counting results. For $\bp\in \bZ^n$, $\bq\in \bZ^m$ and $A\in\textrm{Mat}_{m\times n}(\bK)$, consider the inequality
\begin{equation}\label{multi}\|A\bq -\bp\|_\infty<\phi(\|\bq\|_\infty),\quad1 \leq \| \bq \|_\infty \leq q^T \end{equation}
 where $\| \cdot \|_\infty$ is the supremum norm.
It was shown in \cite{dkl05} that the number of vectors $(\bp,\bq)\in\bZ^n \times \bZ^m$ satisfying the inequality (\ref{multi}) is
$$\Phi(T)+O\left(\Phi(T)^{\frac{1}{2}}\log^{\frac{3}{2}+\epsilon}(\Phi(T))\right)$$ for $\lambda^{mn}$-almost every $A\in \textrm{Mat}_{m\times n}(\bK)$, where $\Phi(T)=\underset{ \bq \in \bZ, \;  \|\bq\|_\infty\le q^T}{\sum}\phi(\|\bq\|_\infty)^n$. 

 Generalizing this result, we prove the following weighted version.
Consider the weighted quasi-norms given by
$$\|\bx\|_\alpha=\max_{1\le i\le m}|x_i|^{1/a_i}\qquad\textrm{and}\qquad\|\by\|_\beta=\max_{1\le j\le n}|y_j|^{1/a_{j+m}}$$ for $\bx\in\bK^m$ and $\by\in\bK^n$. Denote by ${N}_R(T,A)$ the number of nonzero vectors $(\bp,\bq) \in \bZ^m \times \bZ^n$ satisfying the following Diophantine inequalities
\begin{equation}\label{daw0}\|A\bq-\bp\|_\alpha<\frac{q^R}{\|\bq\|_\beta},\qquad 1\le\|\bq\|_\beta\le q^T.\end{equation}

\begin{prop}\label{prop:1.2}   There is a function $\displaystyle\Psi_R\colon \mathbb{N}\to\mathbb{R}$ such that $$N_R(T,A)=\Psi_R(T)+O\left(\Psi_R(T)^{\frac{1}{2}}\log^{2+\epsilon}(\Psi_R(T))\right)$$ for $\lambda^{mn}$-almost every $A\in \textrm{Mat}_{m\times n}(\bK)$. 
\end{prop}

For the precise formula of $\Psi_R$, see Theorem~\ref{counting}.

Using Proposition~\ref{prop:1.2}, we obtain an ergodic theorem for certain unbounded functions as appeared in \cite{emm98} and \cite{ksw16}. Every $\mathbf{Z}$-submodule $\Delta$ of rank $r$ in $\mathbf{K}^d$ has a $\mathbf{Z}$-basis. For $L=\mathbf{K}\Delta$, we denote by $\textrm{vol}(L/\Delta)$ the covolume of $\Delta$ in $L$, normalized by $\textrm{vol}(\mathbf{K}^r/\mathbf{Z}^r)=1$ for any $r\in\mathbb{N}$.

In particular, we call a $\bZ$-submodule $\Lambda$ of rank $d$ a $\mathbf{Z}$-\emph{lattice} and a $\mathbf{Z}$-lattice is \emph{unimodular} if its covolume in $\mathbf{K}^d$ is 1. Since $SL(d,\bK)$ acts transitively on the space of unimodular lattices in $\bK^{d}$ and the stabilizer of $\bZ^{d}$ is $SL(d,\bZ)$, we can identify the space with $X=SL(d, \bK) / SL(d, \bZ)$.
For a unimodular lattice $\Lambda\subset \bK^d$ and a $\bZ$-submodule $\Delta\le\Lambda$, let us define $$\alpha(\Lambda):=\max\{\textrm{vol}(\mathbf{K}\Delta/\Delta)^{-1}\colon \Delta\le\Lambda\}.$$ 
Let $C_\alpha(X)$ be the space of functions $\varphi$ on $X$ such that
\begin{enumerate}
\item[(1)] The function $\varphi\colon X\to\mathbb{R}$ is continuous except on a $\mu$-null set, and
\item[(2)] There exists $C>0$ such that for all $\Lambda\in X$, we have $|\varphi(\Lambda)|\le C\alpha(\Lambda)$.
\end{enumerate}

\begin{thm}\label{ewu} Let $x_0$ be the point $\Gamma$ in $X$. For all $\varphi\in C_\alpha(X)$ and $\lambda_H$-almost every $h\in H$, we have $$\lim_{N\to\infty}\frac{1}{N}\sum_{n=0}^{N-1} \varphi(g_\ba^n hx_0)=\int_{X}\varphi d\mu.$$
\end{thm}

Finally, using Theorem~\ref{ewu}, we obtain a more general counting result concerning the direction of vectors, analogous to \cite{apt16} and \cite{ksw16}. Let $\bS^{m-1}:=\{\bx\in\bK^m\colon \|\bx\|_\infty=1\}$ and $\pi_\alpha\colon \bK^m\to\bS^{m-1}$ be the map given by $\pi_\alpha(\bx)=\{(t^{a_1}x_1,\cdots,t^{a_m}x_m)\}\cap\bS^{m-1}$. Similarly, let $\pi_\beta\colon\bK^n\to\bS^{n-1}$ be given by $\pi_\beta(\bx)=\{(t^{a_{m+1}}x_1,\cdots,t^{a_{m+n}}x_n)\}\cap\bS^{n-1}$. Given measurable subsets $C_1\subset \bS^{m-1}$ and $C_2\subset\bS^{n-1}$, let us define the set $E_{T,R}(C_1,C_2)$ by
$$
\left\{(\bx,\by)\in\bK^m\times\bK^n\colon \|\bx\|_\alpha<\frac{q^R}{\|\by\|_\beta},1\le\|\by\|_\beta\le q^T,\pi_\alpha(\bx)\in C_1, \pi_\beta(\by)\in C_2\right\}.$$
Denote by $N_R(T,A)(C_1,C_2)$ be the number of nonzero solutions of
(\ref{daw0}) satisfying $\pi_\alpha(A\bq-\bp)\in C_1$ and $\pi_\beta(\bq)\in C_2$. The Haar measure on $\bK^m\times\bK^n$ is $\lambda^{m+n}$. Let us denote the measures on $\bS^{m-1}, \bS^{n-1}$ induced from $\bK^m, \bK^n$ by $\lambda_{\bS^{m-1}}, \lambda_{\bS^{n-1}}$, respectively.

\begin{thm}[Counting with directions]\label{cwd} Let $C_1$ and $C_2$ are measurable subsets of $\bS^{m-1}$ and $\bS^{n-1}$, respectively, with measure zero boundaries with respect to $\lambda_{\bS^{m-1}}, \lambda_{\bS^{m-1}}$, respectively.
As $T\to\infty$,
$$N_R(T,A)(C_1,C_2)\sim \lambda^{m+n}(E_{T,R}(C_1,C_2))$$ holds for $\lambda^{mn}$-almost every $A\in\textrm{Mat}_{m\times n}(\bK)$. \end{thm}


The structure of this paper is as follows: Section 2 is devoted to prove the uniform quantitative non-divergence of $H$-orbits. In section 3, we prove an effective double equidistribution and use it to obtain the pointwise equidistribution for compactly supported continuous functions. We give an asymptotic formula with an error rate for the number of solutions to the Diophantine approximation with weights in section 4. Using this result, we show in section 5 the pointwise equidistribution for unbounded functions contained in $C_\alpha(X)$.

\vspace{.1 in}
\textit{Acknowledgement.} We would like to thank Ronggang Shi for helpful discussions and the anonymous referee for valuable comments. We would also like to thank KIAS where part of this work was done. The work is supported by Samsung Science and Technology Foundation under Project No. SSTF-BA1601-03.


\section{Quantitative non-divergence}
Throughout the article, we will mostly use the supremum norm $$\|\bx\|_\infty=\underset{1\le i\le n}{\max}|x_i|$$ for $\bx=(x_1,\cdots,x_d)\in\bK^d$. 

For each $\bZ$-submodule $\Lambda \subset \bK^d$ of rank $r$, we can choose a basis $\bv_1,\bv_2,\cdots,\bv_r$ and define the norm $\|\cdot\|$ by
$$\|\Lambda\|=\|\bv_1\wedge \bv_2\wedge\cdots\wedge\bv_r\|_\infty.$$
This norm is independent of the choice of the basis.
A $\bZ$-submodule $\Delta$ of $\bZ^d$  is called \emph{primitive} if $\Delta=\bK\Delta\cap\bZ^d$.

Also, let $\delta\colon X\to\mathbb{R}^+$ be the length of shortest vector, i.e. $$\delta(\Lambda)=\underset{\bv\in\Lambda\backslash\{0\}}{\inf}\|\bv\|_\infty.$$ 
If $\delta(g\Gamma)<\epsilon$, then there exists a non-zero vector $\bw\in \bZ^d$ such that $\|g\bw\|_\infty<\epsilon$.


\begin{de}[Good functions] Let $\nu$ be a locally finite Borel measure on metric space $X$. For given $C>0, \alpha>0$, a function $f\colon V\subset X\to \bK$ (or $\mathbb{R}$)
 is \emph{$(C,\alpha)$-good} on $V$ with respect to $\nu$ if for any open ball $\bB\subset V$ and for any $\epsilon>0$, we have
$$\nu(\{\bv\in \bB\,|\, |f(\bv)|<\epsilon\}|\le C\left(\frac{\epsilon}{\sup_\bB |f|}\right)^\alpha\nu(\bB).$$
\end{de}

\begin{lem}[\cite{kt07}, Lemma 2.4]\label{Polynomials} For any $r,s\in\N$, there exists a constant $C=C(r,s)$ such that every polynomial $f\in\bK[x_1,\cdots,x_r]$ of degree $\deg(f) \leq s$ is $(C,\frac{1}{rs})$-good on $\bK^r$ with respect to the Haar measure $\lambda^r$ on $\bK^r$.
\end{lem}

Let $X_\epsilon=\{\Lambda\in X\,|\,\delta(\Lambda)\ge\epsilon\}$. By Mahler's compactness criterion, it follows that $X_\epsilon$ is compact. 
\begin{thm}[\cite{gh07}, Theorem 4.5]\label{Ghosh} Let $d,l\in\N, C,\alpha>0$ and $0<\rho<1$ be given. Let $\bB\subset \bK^l$ be an open ball and $h\colon \bB\to GL(d,\bK)$ a continuous map. For any nonzero primitive submodule $\Delta$ of $\bZ^d$, let $\psi_\Delta(A)=\|h(A)\Delta\|$. Assume that
\begin{enumerate}\setlength\itemsep{-\parsep}
\item[(1)] The function $\psi_\Delta\colon \bB\to\mathbb{R}$ is $(C,\alpha)$-good on $\bB$ {for every $\Delta$}.
\item[(2)] $\sup_{\bB} \psi_\Delta\ge\rho$ {for every $\Delta$}.
\item[(3)] For each $A\in \bB$, there are only finitely many $\Delta$ for which $\psi_\Delta(A)<\rho$.
\end{enumerate}
Then for any $0<\epsilon<\rho$, we have
$$\lambda^l(\{A\in \bB\,|\,\delta(h(A)\bZ^d)<\epsilon\})\le dC\left(\frac{\epsilon}{\rho}\right)^\alpha\lambda^l(\bB),$$
where $\lambda^l$ is the Haar measure on $\bK^l$.
\end{thm}

Let us first check the condition (2) of Theorem~\ref{Ghosh} for $h(A) = g_\ba u_A g$, where $u_A$ is as in (\ref{uA}). 
Remark that the lower bound in condition (2) is expressed in terms of the minimum of $a_i$'s. Denote $\min \ba:=\underset{1 \leq i \leq d}{\min} \{a_i\}.$
\begin{lem}\label{Uniformity} There exists $\beta>0$ with the following property. Let $\bB$ be an open $\| \cdot \|_\infty$- ball centered at $e$ in $M$. Then for each $r=1,\cdots,d-1$, $\bv\in V \overset{\textrm{def}}{=} \bigwedge^r\bK^{d}$ and $\ba\in\mathfrak{a}^+$, one has
$$\sup_{A\in\bB}|g_\ba u_A\bv|\ge bq^{\beta\min \ba}|\bv|$$ for some constant $b$ depending on the ball $\bB$ and independent of $\bv$.
\end{lem}
\begin{proof} We follow the idea of Lemma 5.1 in \cite{sh96}. Let $\sigma\colon G\to GL(V)$ be the exterior representation. Let $W=\{\bv\in V\colon \sigma(u_A)\bv=\bv\textrm{ for all }A \in M\}$ and denote by $p_W$ the projection $V\to W$ as a vector space.

The $H$-invariant subspace $W$ can be described as follows: if $r\le m$, then $$W=\textrm{span}\{\vec{e}_{s_1}\wedge\cdots\wedge\vec{e}_{S_k}\,|\,1\le s_1<\cdots<S_k\le m\}$$ and if $r>m$, then
$$W=\textrm{span}\{\vec{e}_{1}\wedge\cdots\wedge\vec{e}_{m}\wedge\vec{e}_{s_1}\wedge\cdots\wedge\vec{e}_{s_{r-m}}\,|\,m+1\le s_1<\cdots<s_{r-m}\le n\}.$$

Let $\{\mathbf{f}_{kl}\}_{1\le k\le m,1\le l\le n}$ be the standard basis of $M$. From the definition, we have $$\sigma(u_{\mathbf{f}_{kl}})\vec{e}_{s_1}\wedge\cdots\wedge\vec{e}_{S_k}=(\vec{e}_{s_1}+\delta_{(l+m),s_1}\vec{e}_{S_k})\wedge \cdots\wedge(\vec{e}_{S_k}+\delta_{(l+m),S_k}\vec{e}_{S_k})$$ and hence it follows that $(\sigma(u_{\mathbf{f}_{kl}})-Id)^2=0$ for each $(k,l)$. Let $\mathcal{I}=\{I=(i_{11},\cdots,i_{mn})\colon i_{kl}\in\{0,1\}\textrm{ for }1\le k\le m\textrm{ and }1\le l\le n\}$. 

For each $I=(i_{11},\cdots,i_{mn}) \in \mathcal{I}$ and $$A=\left(
\begin{array}{cccc}
y_{11} & y_{12} & \cdots & y_{1n} \\
y_{21} & y_{22} & \cdots & y_{2n} \\
\vdots & \vdots & \ddots & \vdots   \\
y_{m1} & y_{m2} & \cdots & y_{mn}
\end{array}\right),$$ let us define $$\tau(u_A^I)=\prod_{k=1}^m\prod_{l=1}^n\left(\sigma(u_{y_{kl}}\mathbf{f}_{kl})-Id\right)^{i_{kl}}.$$

Let $F=\sum_{k=1}^m\sum_{l=1}^n\mathbf{f}_{kl}$ and $T\colon V \to \underset{I\in\mathcal{I}}{\bigoplus} W$ be the linear transformation defined by $$T(\bv)=[p_W(\tau(u_{F}^I)\bv)]_{I\in\mathcal{I}}.$$
From the description of $W$, for each $\bv\ne 0$ in $V$ we can choose $I\in\mathcal{I}$ such that $\tau(u_F^I)\bv\ne W^\perp$.
Therefore, $T$ is injective.

For each $1\le k\le m$ and $1\le l\le n$, fix $x_{kl,0}$ and $x_{kl,1}$ in $\bK$ such that $x_{kl,i}\mathbf{f}_{kl}\in \bB$ and $|x_{kl,0}|\ne|x_{kl,1}|$. Take $\bx_I=\sum_{k=1}^{m}\sum_{l=1}^n x_{kl,i_{kl}}\mathbf{f}_{kl} \in \bB$ and
let us define $S : V \to  \underset{I\in\mathcal{I}}{\bigoplus} W$ by
 $$S (\bv)=[p_W(\sigma(u_{\bx_I})\bv)]_{I\in\mathcal{I}}.$$
Then $S (\bv)=Q \circ T(v)$ with the invertible linear transformation $$Q \colon \underset{I\in\mathcal{I}}{\bigoplus} W\to \underset{I\in\mathcal{I}}{\bigoplus} W$$ defined by
$$ Q_{I J} = x_{1 , i_1}^{j_1} \cdots x_{mn , i_{mn}}^{j_{mn}}Id_W:W\to W.$$ Note that
 $|\det(Q)|=\prod_{k=1}^{mn}|(x_{k,0}-x_{k,1})|^{2\textrm{dim}W}$. 

By injectivity of $S$, there exists $c_0>0$ such that 
\begin{align*}
& q^{\min \ba}\|\bv\|_\infty \le c_0q^{\min \ba}\|S(\bv)\|_\infty \le c_0 q^{\min \ba}\sup_{A\in\bB}\|p_W(\sigma(u_A)\bv)\|_\infty \\ &\le c_0\sup_{A\in\bB}\|p_W(\sigma(g_\ba u_A)\bv)\|_\infty \le c_0\sup_{A\in\bB}\|\sigma(g_\ba u_A)\bv\|_\infty 
\end{align*} which completes the proof. 
\end{proof}

\begin{ex} Let $G=SL(3,\bK)$ with $m=1$ and $n=2$, $V=\bigwedge^2\bK^3$ and $\sigma\colon G\to GL(V)$. In this case, $W=\bK\vec{e}_1\wedge\vec{e}_2+\bK\vec{e}_1\wedge\vec{e}_3$. For 
$$u_A=\left(
\begin{array}{ccc}
1 & y_1 & y_2 \\
0 & 1 & 0 \\
0 & 0 & 1  
\end{array}\right),
\qquad \sigma(u_A)=\left(
\begin{array}{ccc}
1 & 0 & -y_2 \\
0 & 1 & y_1 \\
0 & 0 & 1 
\end{array}\right).$$
For a vector $\bv$ in $V$ given by $a\vec{e_1}\wedge\vec{e_2}+b\vec{e_1}\wedge\vec{e_3}+c\vec{e_2}\wedge\vec{e}_3$, we have $$T(v)=\left[\begin{array}{c}
a\vec{e}_1\wedge\vec{e}_2+b\vec{e}_1\wedge\vec{e}_3\\
c\vec{e}_1\wedge\vec{e}_3\\
-c\vec{e}_1\wedge\vec{e}_2\\
0
\end{array}\right]$$ 
and
$$S(v)=\left[\begin{array}{c}
(a-x_{2,0}c)\vec{e}_1\wedge\vec{e}_2+(b+x_{1,0}c)\vec{e}_1\wedge\vec{e}_3\\
(a-x_{2,0}c)\vec{e}_1\wedge\vec{e}_2+(b+x_{1,1}c)\vec{e}_1\wedge\vec{e}_3\\
(a-x_{2,1}c)\vec{e}_1\wedge\vec{e}_2+(b+x_{1,0}c)\vec{e}_1\wedge\vec{e}_3\\
(a-x_{2,1}c)\vec{e}_1\wedge\vec{e}_2+(b+x_{1,1}c)\vec{e}_1\wedge\vec{e}_3
\end{array}\right].$$ 
Then $S(v)=Q\circ T(v)$ with the invertible linear transformation $Q\colon \underset{I\in\mathcal{I}}{\bigoplus} W\to \underset{I\in\mathcal{I}}{\bigoplus} W$ given by $$\left(\begin{array}{cccc}
Id_W & x_{1,0}Id_W & x_{2,0}Id_W & x_{1,0}x_{2,0}Id_W \\
Id_W & x_{1,1}Id_W & x_{2,0}Id_W & x_{1,1}x_{2,0}Id_W \\
Id_W & x_{1,0}Id_W & x_{2,1}Id_W & x_{1,0}x_{2,1}Id_W \\
Id_W & x_{1,1}Id_W & x_{2,1}Id_W & x_{1,1}x_{2,1}Id_W 
\end{array}\right).$$ In this case, $|\det(Q)|=|(x_{1,0}-x_{1,1})|^4|(x_{2,0}-x_{2,1})|^4$. 

\end{ex}

\begin{cor}\label{Nondivergence} Let $L$ be a compact subset of $X$ and $\bB\subset H$ be a ball centered at $e\subset H$. Then there exists $N=N(\bB,L)$ such that for every $0<\epsilon<1$,
 any $x=g\Gamma\in L$, and any $\ba\in\mathfrak{a}^+$ with $\min \ba\ge N$ we have
$$\lambda_H (\{A\in\bB \,|\, \delta(g_\ba u_A x)< \epsilon\})\ll\epsilon^{1/mn(m+n)}\lambda_H(\bB),$$
where the implied constant depends on $\bB$ and $L$.
\end{cor}
\begin{proof}  
Let $h\colon \bB \to G$ be given by $h(A)=g_{\ba}u_A g$. For any primitive submodule $\Delta \subset \bZ^d$ of rank $r$ for $1 \leq r \leq m+n$, there is a constant $C=C(r) >0$ such that $A\mapsto\|h(A)\Delta\|$ is $(C(r) ,\frac{1}{mnr})$-good by Lemma~\ref{Polynomials}. By letting $C = \min_{r=1, \cdots, m+n} \{ C(r) \}$ and $\alpha = \frac{1}{mn(m+n)}$, for any $\Delta$, $A\mapsto\|h(A)\Delta\|$ is $(C , \alpha)$-good.

 By compactness of $L$ and discreteness of $\bigwedge^r\bZ^d$ in $\bigwedge^r\bK^d$,
$$\inf\{\|g\bv\|_\infty \colon g\Gamma\in L, \bv\in\bigwedge\nolimits^r\bZ^d-\{0\},r=1,\cdots,d-1\}$$
is positive, say $c_1$. Together with Lemma \ref{Uniformity}, we can find $c>0$ such that for any $g\in \pi^{-1}(L)$ we have $\sup_{A\in\bB}\|g_\ba u_A g\bv\|_\infty \ge cq^{\min\ba}\|g\bv\|_\infty \geq c q^N c_1 $. Thus, the conditions of Theorem~\ref{Ghosh} are satisfied with $\rho=c q^N c_1$. 
\end{proof}


\section{Pointwise equidistribution with an error rate}
In this section, we prove Theorem \ref{pe} which gives the pointwise equidistribution for compactly supported functions in each $H$-orbit. Before proving Theorem \ref{pe}, we first prove effective equidistribution of $g_\ba$-translates of $H$-orbits, a positive characteristic analog of a theorem of Kleinbock-Margulis \cite{km12}. For the case of \emph{equal weights}, i.e. when $\ba=(a,\cdots,a,b,\cdots,b)$ such that $ma=nb$, see also \cite{mo11}. 

For a given $\phi\in L^2(X)$, we denote its $L^2$-norm by $\|\phi\|_2$. Let $K$ be the maximal compact subgroup $SL(n,\mathcal{O})$ of $G$ and $K^{(l)}$ be the $l$-th congruence subgroup $(Id+t^{-l}\textrm{Mat}(n,\mathcal{O}))\cap K$ of $K$. Following \cite{emmv15} for the adelic case (and R\"{u}hr's thesis \cite{ru15} for $p$-adic case), let us define the degree-$k$ Sobolev norm $S_k(\phi)$ of $\phi$ as follows.
Let $Av_l$ be the averaging projection to the set of $K^{(l)}$-invariant functions, i.e. $\displaystyle Av_l(\phi)(g)=\int_{K^{(l)}}\phi(gk)dk$ where $dk$ is the probability Haar measure on $K^{(l)}$. Let $\textrm{pr}[l]$ be the \emph{$l$-th level projection} $Av_l-Av_{l-1}$ for $l\ge1$ and $\textrm{pr}[0]=Av_0$. Define
$$S_k(\phi)^2\overset{\textrm{def}}{=}\sum_{l\geq 0} q^{lk}\|\textrm{pr}[l](\phi)\|_2^2.$$
For example, we have \begin{align*}S_k(\chi_{K^{(l)}})=&\left(\frac{1}{[K:K^{(l)}]}\right)^2+\sum_{j=1}^lq^{jk}\left[\left(\frac{\lambda(K^{(j-1)}-K^{(j)})}{[K^{(j-1)}:K^{(l)}]}\right)^2+\left(\frac{\lambda(K^{(j)})}{[K^{(j)}:K^{(l)}]}\right)^2\right]\\ \asymp& q^{(k-\textrm{dim}G)l}.\end{align*}
The following properties of the Sobolev norm will be used in the sequel.
\begin{lem}[Properties of Sobolev norm]\label{sobolev} Given $\phi\in C_c^\infty(\mathbf{K}^N)$, we have
\begin{enumerate}
\item $S_{N+1}$ dominates the supremum norm, i.e. $\|\phi\|_\infty\ll S_{N+1}(\phi)$.
\item For a given $\psi_1\in C_c^{\infty}(\mathbf{K}^N)$, if we define $\theta\in C_c^{\infty}(\mathbf{K}^N)$ by $\theta_1(x)=\phi(x)\psi_1(x)$, then $S_k(\theta_1)\ll S_k(\phi)S_k(\psi_1)$. 
\item For a given $\psi_2\in C_c^{\infty}(\mathbf{K}^M)$, if we define $\theta\in C_c^{\infty}(\mathbf{K}^{N+M})$ by $\theta_2(x_1,x_2)=\phi(x_1)\psi_2(x_2)$, then $S_k(\theta_2)\ll S_k(\phi)S_k(\psi_2)$.
\end{enumerate}
\end{lem}
The key idea of the proof is to use Cauchy-Schwartz inequality, after decomposing $\phi$ as $\sum_{l\ge 0}\textrm{pr}[l]\phi$. We omit the details of the proof, which is analogous to \cite{emmv15} and \cite{ru15}.
 
Using arguments of \cite{emmv15}, one can easily obtain the following effective decay of matrix coefficients of the regular representation.
\begin{prop}\label{Exp mixing}
Given any functions $\phi$ and $\psi$ in $C_c^\infty(X)$, there exists $\delta_0>0$ such that
$$\left|\int_{X}\phi(x)\psi(g_{\ba}x)d\mu(x)-\int_{X}\phi d\mu\int_{X}\psi d\mu\right|\le C_0S_{d^2}(\phi)S_{d^2}(\psi)e^{-\delta_0\min \ba}$$ holds for some $C_0>0$ which depends only on the group $G$ and $\Gamma$.
\end{prop}
\begin{proof}
It follows from Theorem 2.1 of \cite{agp12} (see also \cite{oh02} for the setting more general than the function field case) that there exists $\delta_0>0$ such that
$$\left|\int_{X}\phi(x)\psi(g_{\ba}x)d\mu(x)-\int_{X}\phi d\mu\int_{X}\psi d\mu\right|\le C\|\phi\|_2\|\psi\|_2e^{-\delta_0\min \ba}$$ holds for some $C>0$ depending on the smoothness of the functions $\phi$ and $\psi$, i.e. the choice of compact subgroup $U$ in the definition of the smooth function. By replacing the $L^2$-norm above by a Sobolev norm of degree $d^2$ which is equal to $\textrm{dim}\,G+1$, we apply the argument of subsection A.8 of \cite{emmv15} (or Proposition 3.2.1 of \cite{ru15}) to conclude that the implied constant in the statement of the Proposition depends only on $G$ and $\Gamma$. 
\end{proof}
Using Proposition~\ref{Exp mixing}, we prove the effective equidistribution of expanding translates of $H$-orbits.
In the proof of the theorem, we will use the decomposition of $G$ into $$G=H^- H^0 H,$$ where
$H^-=\left\{
\begin{psmallmatrix} 
I_m & 0 \\
L & I_n 
\end{psmallmatrix} 
\colon L\in\textrm{Mat}_{n\times m}(\bK)\right\}$ and
$$H^0=\left\{
\begin{psmallmatrix} 
P & 0 \\
0 & Q 
\end{psmallmatrix} 
: P\in GL(m,\bK),Q\in GL(n,\bK), \det(P)\det(Q)=1\right\}.$$ 
It is easy to check that the decomposition is unique, more precisely,
$$\begin{psmallmatrix} 
A & B \\
C & D 
\end{psmallmatrix}  = \begin{psmallmatrix} 
1 & 0 \\
CA^{-1} & 1 
\end{psmallmatrix} \begin{psmallmatrix} 
A& 0 \\
0 & D-CA^{-1}B 
\end{psmallmatrix} \begin{psmallmatrix}
1 & A^{-1}B \\
0 & 1 
\end{psmallmatrix} .
$$
Before proving the effective equidistribution for general case, we present the proof of the following special case.
\begin{prop}\label{equalweight} Let $\be=(n,\ldots,n,m,\ldots,m)\in\mathfrak{a}^+$. Let $f\in C_c^{\infty}(H)$, $r>0$ and $x\in X$ be such that supp$\,f\subset B_H(r)$ and $\pi_x$ is injective on $B_G(2r)$. Then for any $\phi\in C_c^{\infty}(X)$ with $\int_X \phi=0$, there exists $E=E(r,\phi)$ such that for any $i\in\mathbb{N}$ we have
$$\left|\int_H f(h)\phi(g_\be^ihx)d\lambda_H(h)\right|\le ES_{d^2}(f)S_{d^2}(\phi)e^{-\delta_0i\min \be}$$ with the same $\delta_0$ as in Proposition~\ref{Exp mixing}.
\end{prop}
\begin{proof} Note that there is a level $l$ congruence subgroup $K^{(l)}$ of $K$ such that $f$ and $\phi$ are invariant under $K^{(l)}\cap H$ and $K^{(l)}$, respectively. Taking $l$ large enough, we may assume that $K^{(l)}\subset B_G(\frac{r}{2})$. Choose $$f^{-}=\frac{1}{\lambda_{H^-}(K^{(l)}\cap H^-)}\chi_{K^{(l)}\cap H^-}\textrm{ and }f^0=\frac{1}{\lambda_{H^0}(K^{(l)}\cap H^0)}\chi_{K^{(l)}\cap H^0}.$$ If we define $\widetilde{f}(h^-h^0)=f^-(h^-)f^0(h^0)$, then $\textrm{supp}\,\widetilde{f}\subset B_G(r)$ and $S_{d^2}(\widetilde{f})\asymp q^{(d^2-\textrm{dim} H^-H^0)l}=q^{(mn+1)l}$. Define $\varphi\in C_c^{\infty}(X)$ by $\varphi(h^-h^0hx)=\widetilde{f}(h^-h^0)f(h)$, which is supported on the small neighborhood of $x\in X$. This definition makes sense because $\pi_x$ is injective on $B_G(2r)$. Since the modular function on $K^{(l)}\cap H^0$ is trivial and $g_\be^ih^-h^0g_\be^{-i}\subset K^{(l)}$ when $h^-h^0\in K^{(l)}$, we have
\begin{align*}
&\left|\int_H  f(h)\phi(g_{\be}^i hx) d\lambda_H(h)- \int_Gf^-(h^-)f^0(h^0)f(h)\phi(g_\be^ih^-h^0hx)d(h^-h^0h)\right| \\
=&\left|\int_G  f^-(h^-)f^0(h^0)f(h)[\phi(g_{\be}^ihx)-\phi(g_\be^ih^-h^0g_\be^{-i}g_\be^ihx)]d(h^-h^0h)\right|=0.
\end{align*}
Therefore, \begin{align*}&\left|\int_H  f(h)\phi(g_{\be}^i hx) d\lambda_H(h)\right|=\left|\int_X\varphi(gx)\phi(g_\be^igx)d\mu(g)\right|\\
\le&C_0S_{d^2}(\varphi)S_{d^2}(\phi)e^{-\delta_0 i\min \be}\le C_0q^{(mn+1)l}S_{d^2}(f)S_{d^2}(\phi)e^{-\delta_0i\min \be}
\end{align*}
by Proposition~\ref{Exp mixing} and Lemma~\ref{sobolev}, (3). Since $l$ depends only on $r$ and the smoothness of $\phi$, this completes the proof.
\end{proof}
Using the above proposition, we will prove the general case. Let us define the integral vector $\overline {\mathbf{a}}$ of equal weight
$\bs= (\min \ba)\be$ where $\be=(n,\ldots,n,m\ldots,m)$.
Note that for $i$ such that $\ell-1 \leq \frac{i}{2(m+n)} < \ell$ for some $\ell \in \mathbb{N}$,
$\bt\in\mathfrak{a}^+$ and $\min (\bt)\ge \min \ba i/2$. 
\begin{thm}[Effective equidistribution of expanding translates]\label{Effective} There exists $\delta_1>0$ such that the following holds. For any $f\in C_c^\infty(H)$, $\phi\in C_c^\infty(X)$ and for any compact $L\subset X$, there exists $C_1=C_1(f,\phi,L)$ such that for all $x\in L$ and $i \in\mathbb{N}$, we have
$$\left|\int_H f(h)\phi(g_\ba^i hx)d\lambda_H(h)-\int_H f(h)d\lambda_H(h)\int_{X} \phi(x)d\mu(x)\right|\le C_1e^{-\delta_1 i \min \ba }.$$
\end{thm}
\begin{proof} By choosing large enough $C_1$, it is enough to show the proposition for large enough $i$. The proof is analogous to the argument of \cite{km12}.
Throughout the proof, we denote by $B_H(r), B_G(r)$ the ball of radius $r$ centered at $e$ in $H, G$, respectively.

Without loss of generality, we may assume that $\int_{X} \phi \,d\mu=0$ and $\phi$ is invariant under an open compact subgroup $K_\phi$ of $G$. Let $\kappa>0$ be small enough so that $K_\phi\supset B_G(q^{-\kappa i})$, and set $\chi=\frac{1}{\lambda_H(B_H(q^{-\kappa i}))}\chi_{B_H(q^{-\kappa i})}$.
Choose a ball $B=B_H(r_1)$ containing the support of $f$ and   
 enlarge the ball $B$ slightly to obtain $\widetilde{B}=B_H(r_1+q^{-2\min(\bt)} q^{-\kappa i})$. 

We choose $i_0$ large enough so that for all $i \geq i_0$, we have $\lambda_H(\widetilde{B})\le 2\lambda_H(B)$ and $\min (\bt) \ge N(\widetilde{B},L)$ for the constant $N(\widetilde{B}, L)$ of Corollary \ref{Nondivergence}. It follows that given any $\epsilon>0$, any $x \in L$ and $V_1^{\epsilon}:=\{h\in \widetilde{B}\,|\,\delta(g_{\bt} h x)<\epsilon\},$ we have
$\lambda_H (V_1^\epsilon) \ll \epsilon^{1/(mn(m+n))} \lambda_H(\widetilde{B}) \leq 2\epsilon^{1/(mn(m+n))} \lambda_H(B).$
 Therefore it holds that
\begin{align*}&\int_{V_1^{\epsilon}} f(h)\phi(g_{\ba i} hx)d\lambda_H(h)
\le \lambda_H(V_1^{\epsilon})\|f\|_\infty\|\phi\|_\infty \ll \epsilon^{\frac{1}{mnd}}\lambda_H(B)S_{d^2}(f) S_{d^2} (\phi) ,
\end{align*}
for the degree-$d^2$ Sobolev norm by Lemma~\ref{sobolev}, (1).

Now let us compute the integral on the complement $V_2^{\epsilon}:=\widetilde{B}\backslash V_1^{\epsilon}$. Given any $y\in B_H(q^{-\kappa i})$, let $k_y=g_{\bt}^{-1}y g_\bt \in H$. Since $d(e,k_y)\le q^{-2\min(\bt)}d(e,y)$, we have $k_y \in B_H(q^{-2\min (\bt)} q^{-\kappa i}).$
 Thus, the support of all functions of the form $h\mapsto f(kh)$ is contained in $\widetilde{B}$.

If we define $f_h(y)=f(k_y h) \chi(y)=f(g_\bt^{-1}yg_\bt h)\chi(y)$, then we have
\begin{align*}\displaybreak[0]&\left|\int_{V_2^{\epsilon}} f(h)\phi(g_{\ba}^i hx)d\lambda_H(h)\right|=\left|\int_{V_2^{\epsilon}} f(h)\phi(g_{\ba i} hx)d\lambda_H(h)\int_H \chi(y) d\lambda_H(y)\right|\\\displaybreak[0]
=&\left|\int_{V_2^{\epsilon}} \int_H f(k_yh)\chi(y)\phi(g_{\bs \ell} yg_\bt hx) d\lambda_H(y)d\lambda_H(h)\right| \\\displaybreak[0]
\le&\int_{V_2^{\epsilon}}\left|\int_Hf_{h}(y)\phi(g_{\bs \ell} yg_\bt hx)d\lambda_H(y)\right|d\lambda_H(h).
\end{align*}

We apply Proposition~\ref{equalweight} with $q^{-\kappa i}$ in place of $r$, $f_h$ in place of $f$ and $g_\bt hx$ in place of $x$.
Note that the injectivity radius of $X_\epsilon$ is at least $c_0\epsilon^d$ for some $c_0>0$ (See the argument of Proposition 3.5 in \cite{km12}). We choose $\epsilon=(\frac{2}{c_0})^{1/d} q ^{-\kappa i/d}$ so that for $h\in V_2^{\epsilon}$, the projection map $\pi= \pi_{g_\bt hx}\colon G\to X$ given by $\pi_{g_\bt hx}(g)=gg_\bt hx$ is injective for $g \in B_G(2q^{-\kappa i})$. Thus, all the assumptions of Proposition~\ref{equalweight} holds and we have
\begin{align*}
&\int_{V_2^{\epsilon}}\left|\int_H  f_{h}(y)\phi(g_{\bs \ell} yg_\bt hx) d\lambda_H(y) \right| d\lambda_H(h) \\ \ll&S_{d^2}(\phi) E(q^{-\kappa i},\phi)e^{-\delta_0 \ell \min \bs}\int_{V_2^{\epsilon}}S_{d^2}(f_h)\,d\lambda_H(h)
\end{align*}
for $\delta_0$ in Proposition~\ref{Exp mixing}. Since the mapping $\eta\colon y\mapsto k_y $ is contracting, the $\ell^2$-norm of the $l$-th level projection of $f$ will not increase after precomposition with $\eta$ for each $l\ge 0$. It follows from Lemma~\ref{sobolev} that $S_{d^2}(f_h)\ll S_{d^2}(f)S_{d^2}(\chi)$. Therefore,  
\begin{align*}
\left|\int_H f(h)\phi(g_\ba^i hx)d\lambda_H(h)\right|& \ll q^{-\frac{\kappa i}{mnd^2}}\lambda_H(B)S_{d^2}(f)S_{d^2}(\phi)+\\&q^{-\delta_0 l\min \bs/\log q}\lambda_H(B)S_{d^2}(f)S_{d^2}(\phi)S_{d^2}(\chi)E(q^{-\kappa i},\phi).
\end{align*}
Note that $S_{d^2}(\chi)\ll q^{(m^2+mn+n^2)\kappa i}$ and $E(q^{-\kappa i},\phi)\ll q^{(mn+1)\kappa i}$.
Recall that $\ell-1 \leq \frac{i}{2(m+n)} < \ell$ and hence, if $\kappa>0$ is small enough, then $\delta_0\ell \min\ba/\log q>\kappa i(d^2+1)$. This completes the proof.
\end{proof} 
%


Let us define $$I_{f,\phi,\psi}^{(i,j)}(g_{\ba},x_1,x_2)=\int_H f(h)\phi(g_\ba^ihx_1)\psi(g_\ba^jhx_2)d\lambda_H(h)$$
for $f\in C_c(H)$, $\phi,\psi\in C_c(X)$, $x_1,x_2\in X$ and $i,j\in\mathbb{N}$. The following double equidistribution is the main ingredient of the proof of Theorem~\ref{pe}.

\begin{prop}[Double equidistribution]\label{Double} There exists $\delta_2>0$ with the following property. Given $f\in C_c^\infty(H)$, $\phi,\psi\in C_c^\infty(X)$ and a compact subset $L$ of $X$, there exists $C_2=C_2(f,\phi,\psi,L)$ such that for any $x_1,x_2\in L$ and $i, j\in \mathbb{N}$ with $i\le j$, we have
$$\left|I_{f,\phi,\psi}^{(i,j)}(g_{\ba},x_1,x_2)-\int_H f d\lambda_H(h)\int_{X}\phi d\mu\int_{X}\psi d\mu\right|\le C_2e^{-\delta_2\min(i,j-i)\min \ba}.$$
\end{prop}

\begin{proof}  It is enough to show the statement for sufficiently large $i,j$ and continuous characteristic functions $f,\phi$ and $\psi$.  Choose compact open subgroups $H_f$ of $H$ and $K_\phi$ of $G$ which leave $f$ and $\phi$ invariant, respectively. By taking $l$ large enough, we may assume that the $l$-th congruence subgroup of $K$ satisfies $K^{(l)}\cap H\subset H_f$ and $K^{(l)}\subset K_\phi$. Let $q\in C_c^\infty(H)$ be such that $q\ge 0,\int _Hq(y)dy=1$ and $\supp(q)\subset K^{(l)}\cap H$. Similar to the proof of Theorem~\ref{Effective}, choose $i_0$ large enought so that $k_y:=g_\ba^{-i}yg_\ba^i \in H_f$ for all $i \geq i_0$.
Making a change of variable $h\mapsto k_y h$ gives us 
\begin{align*}
I_{f,\phi,\psi}^{(i,j)}&=\int_H \int_H f(k_y h)\phi(g_\ba^ik_yhx_1)\psi(g_\ba^j k_y hx_2)q(y)dhdy,\\
&=\int_H \int_H f(h)\phi(g_\ba^ihx_1)\psi(g_\ba^j k_yhx_2)q(y)dhdy
\end{align*}
where we used  $K_\phi$-invariance of $\phi$ and the fact that $y\in \supp (q) \subset K_\phi$.

Let $E=\{h\in H_f\,|\, g_\ba^ihx_2\in X_\epsilon\}$. Then for $i \geq N(H_f, L)/\min \ba $, by Corollary~\ref{Nondivergence}, $\lambda_H(H_f\backslash E)\ll \epsilon^\alpha\lambda_H(H_f)$ for $\alpha=\frac{1}{mn(m+n)}$. 
Hence 
$$\left|\int_{H}\int_{H\backslash E} f(h)\phi (g_\ba^ihx_1)\psi(g_\ba^jy_1hx_2)q(y)dhdy\right|\le c\epsilon^\alpha\lambda_H(H_f)\|f\|_\infty\|\phi\|_\infty\|\psi\|_\infty$$ for some $c>0$.

 Meanwhile, for $h\in E$, there exists $C_1=C_1(f,\phi, X_\epsilon)$ such that
\begin{align*}
& \left|\int_H \psi(g_\ba^jk_y hx_2)q(y)dy- \int_X\psi(x_2)d\mu\right| \\
= & \left|\int_H \psi(g_\ba^{j-i}yg_\ba^{i}hx_2)q(y)dy- \int_X\psi(x_2)d\mu\right|
\le C_1e^{-\delta_1 (j-i)\min \ba}
\end{align*}
by applying Theorem~\ref{Effective} with $g_\ba^ihx_2$ in place of $x$, and $X_\epsilon$ in place of $L$.
Finally,
\begin{align*}
&\left|\int_H f(h)\phi(g_\ba^ihx_1)dh-\int_H f\int_X\phi\right|\le C_1e^{-\delta_1i\min \ba}.
\end{align*}
Therefore, for some $r\in\mathbb{N}$, we get
\begin{align*}
\left|I_{f,\phi,\psi}^{(i,j)}(g_{\ba},x_1,x_2)-\int_H f d\lambda_H(h)\int_{X}\phi d\mu\int_{X}\psi d\mu\right|\ll
e^{-\delta_1(j-i)\min \ba}+e^{-\delta_1i\min\ba}
\end{align*}
with an implied constant depending only on the degree-$r$ Sobolev norm of $f$,$\phi$ and $\psi$. This completes the proof of the Proposition.
\end{proof}
\begin{lem}\label{lemma} Let $j\ge i$. Given a probability space $(Y,\nu)$, suppose that $F\colon Y\times\mathbb{Z}_{\ge 0}\to\mathbb{R}$ satisfies $$\left|\int_Y F(y,i)F(y,j)d\nu\right|\le Cq^{-\delta\min (i,j-i)}.$$ Then we have
$$\int_Y\left(\sum_{i=b}^{c-1} F(y,i)\right)^2d\nu(y)\le \frac{4C(c-b)}{1-q^{-\delta}}.$$
\end{lem}
\begin{proof} The proof is verbatim the proof of Lemma 3.2 in \cite{ksw16}. The following inequalities
\begin{align*}
\int_Y &\left(\sum_{i=b}^{c-1} F(y,i)\right)^2d\nu(y)=\sum_{i=b}^{c-1}\sum_{j=b}^{c-1}\int_YF(y,i)F(y,j)d\nu(y)\\
&\le2\sum_{i=b}^{c-1}\sum_{j=i}^{c-1}\left|\int_Y F(y,n)F(y,m)d\nu(y)\right|
\le2C\sum_{i=b}^{c-1}\sum_{j=i}^{c-1}\left(q^{-\delta(j-i)}+q^{-\delta i}\right)\\
&\le2C\sum_{i=b}^{c-1}\left(\frac{1}{1-q^{-\delta}}+q^{-\delta i}(c-i)\right) 
\le\frac{4C(c-b)}{1-q^{-\delta}}
\end{align*} give the conclusion.
\end{proof}

In particular, let $L_s$ be the set of intervals of the form $[2^kl,2^k(l+1))\cap\mathbb{Z}_{\ge 0}$ where $2^k(l+1)<2^s$. Then 
\begin{align*}
\displaybreak[0] \sum_{[b,c)\in L_s}\int_Y\left(\sum_{i=b}^{c-1} F(y,i)\right)^2d\nu(y)&\le\sum_{k=0}^{s-1}\sum_{l=0}^{2^{s-k}-1}\int_Y\left(\sum_{i=2^kl}^{2^k(l+1)-1}F(y,i)\right)^2d\nu(y)\\
\displaybreak[0] &\le\sum_{k=0}^{s-1}\sum_{l=0}^{2^{s-k}-1}\frac{4C2^k}{1-q^{-\delta}}
=\frac{4Cs2^s}{1-q^{-\delta}}.
\end{align*} 

For any given $\epsilon>0$, let $$Y_s=\left\{y\in Y\left|\right.\sum_{I\in L_s}\left(\sum_{n\in I}F(y,i)\right)^2>2^ss^{2+2\epsilon}\right\}.$$
Then $$\nu(Y_s)\le\frac{4Cs^{-(1+2\epsilon)}}{1-q^{-\delta}}$$ and moreover $$\sum_{i=0}^{r-1}F(y,i)\le 2^{s/2}s^{3/2+\epsilon}$$ for $y\notin Y_s$ and for all $1\le r\le2^s$. Borel-Cantelli lemma implies the following corollary. 
\begin{cor}\label{corp} With the notations as above, we have a measurable subset $Z_\epsilon$ of $Y$ with full measure such that for every $y\in Z_\epsilon$ there exists $s_y\in\mathbb{N}$ such that $y\notin Y_s$ whenever $s\ge s_y$.
\end{cor}
\begin{thm}[Pointwise equidistribution]\label{pointwise}
For any given $x\in X$, $\phi\in C^\infty_c(X)$ and $\epsilon>0$, we have
$$\frac{1}{N}\sum_{n=0}^{N-1}\phi(g_\ba^nhx)=\int_X\phi d\mu+O(N^{-1/2}(\log N)^{\frac{3}{2}+\epsilon})$$
for almost every $h\in H$.
\end{thm}
\begin{proof} The proof is similar with Lemma 3.4 and Lemma 3.5 in \cite{ksw16}.
Given $x\in X$, $\phi\in C_c(X)$ and $\epsilon>0$, let $\nu$ be the probability measure on $X$ defined by $$\int\psi d\nu=\int_Hf(h)\psi(hx)d\lambda_H(h).$$
Let $\alpha=\int_X \varphi d\mu$. Then there exist $C>0$ and $\delta>0$ such that $$\left|\int_X (\varphi(g_\ba^nx)-\alpha)(\varphi(g_\ba^nx)-\alpha)d\nu\right|\le Cq^{-\delta\min(n,m-n)}$$ holds for every $m\ge n\ge 0$. Applying Corollary~\ref{corp} to $F(x,n)=\phi(g_\ba^nx)-\alpha$, we obtain a full measure subset $Z_\epsilon$ of $Y$ with the following property. For every $y\in Z_\epsilon$, we have
\begin{align*}
\displaybreak[0] \left|\sum_{n=0}^{N-1} F(y,n)\right|&\le 2^{s/2}s^{3/2+\epsilon}\le (2N)^{1/2}\log^{3/2+\epsilon}(2N)
\end{align*}
when $N\ge 2^{s_y-1}$.
\end{proof}


\section{Diophantine approximation with weights}
As mentioned in the introduction, a quantitative Khintchine-Groshev type theorem over a field of formal series is obtained in \cite{dkl05} as follows. Let $V=\{q^{-n}\colon n\in \mathbb{N}\}$ and $\phi\colon \mathbb{R}^+\to V$. For $\bp\in \bK^n$ and $\bq\in \bK^m$, consider the inequalities
\begin{equation}\label{mda}\|\bq A-\bp\|_\infty<\phi(\|\bq\|_\infty), \quad 1 \leq \| \bq \|_\infty \le q^T.\end{equation}
 Let $\epsilon>0$ be arbitrary and let 
$$\Phi(T)=(q^m-1)\sum_{r=0}^{T} q^{rm}\phi(q^r)^n.$$
The number of solutions $(\bp,\bq) \in \bZ^m \times \bZ^n$ satisfying (\ref{mda}) is
$$\Phi(T)+O\left(\Phi(T)^{1/2}\log^{3/2+\epsilon}(\Phi(T))\right)$$ for $\lambda^{mn}$-almost every $A\in M$.

We consider the weighted quasi-norms
$$\|\bx\|_\alpha=\max_{1\le i\le m}|x_i|^{1/a_i}\qquad\textrm{and}\qquad\|\by\|_\beta=\max_{m+1\le j\le d}|y_j|^{1/a_j}.$$

Define the set
$$E_{T,R}=\left\{(\bx,\by)\in\bK^m\times\bK^n\colon \|\bx\|_\alpha<\frac{q^R}{\|\by\|_\beta},1\le\|\by\|_\beta\le q^T\right\}$$
whose measure $\lambda^{m+n}(E_{T,R})$ is given by 
\begin{align*}\sum_{k\in\mathbb{Q},0\le k\le T}\lambda^m\left(\{\|\bx\|_\alpha<q^{R-k}\}\right)\lambda^n\left(\{\|\by\|_\beta=q^k\}\right).\end{align*} 
Note that
\begin{align*}
 \lambda^n \left(\{\|\by\|_\beta=q^k\}\right) &= \lambda^n \left(\{\|\by\|_\beta\ \leq q^k\}\right) - \lambda\left(\{\|\by\|_\beta < q^k\}\right)\\
 & = \lambda^n \{ \by : \forall j, | y_j | ^{1/a_j} \leq q^k \} - \lambda^n \{ \by : \forall j, | y_j | ^{1/a_j} < q^k \}\\
 &=q^{\lfloor ka_{m+1}\rfloor+\cdots+\lfloor ka_{m+n}\rfloor}-q^{\lceil ka_{m+1}\rceil+\cdots+\lceil ka_{m+n}\rceil-n}
 \end{align*}
Remark that this term is zero if $ka_i$ are not integers for all $i$. 
 Similarly, 
$$ \lambda^m \left(\{\|\bx\|_\alpha<q^{R-k}\}\right) = q^{\lceil(R-k)a_1\rceil+\cdots+\lceil(R-k)a_m\rceil-m}.$$
Thus $\lambda^{m+n}(E_{T,R})$ is equal to 
\begin{align*}
\sum_{k\in\mathbb{Q},0\le k\le T}\left(q^{\lceil(R-k)a_1\rceil+\cdots+\lceil(R-k)a_m\rceil-m}\right)\left(q^{\lfloor ka_{m+1}\rfloor+\cdots+\lfloor ka_{m+n}\rfloor}-q^{\lceil ka_{m+1}\rceil+\cdots+\lceil ka_{m+n}\rceil-n}\right).
\end{align*}

The above sum is in fact a sum over rational numbers of the form $l/a_j$ for some integer $l$ and for some $m+1\le j\le m+n$, thus it is a finite sum and it is well-defined.

We further remark that 
$ E_{t, R} - E_{t-1, R} = E_{1,R}$ for any $t \in \mathbb{N}$, thus it also satisfies
the homogeneity with respect to positive integers $$\lambda(E_{T,R})=T\lambda(E_{1,R}).$$

There is a one-to-one correspondence between the set of nonzero solutions of
\begin{equation}\label{daw}\|A\bq-\bp\|_\alpha<\frac{q^R}{\|\bq\|_\beta},\qquad 1\le\|\bq\|_\beta\le q^T\end{equation} and the intersection of $u_A\Gamma$ with the set $E_{T,R}$. Modifying the argument of \cite{dkl05}, we can compute the number $N_R(T,A)$ of solutions satisfying (\ref{daw}).
\begin{thm}\label{counting} Let $\Psi_R(T)=\lambda^{m+n}(E_{T,R})$. Then, we have $$N_R(T,A)=\Psi_R(T)+O\left(\Psi_R(T)^{\frac{1}{2}}\log^{2+\epsilon}(\Psi_R(T))\right)$$ for $\lambda^{mn}$-almost every $A\in \textrm{Mat}_{m \times n}(\bK)$. 
\end{thm}
\begin{proof} 
Given a vector $\bq\in\bK^m$, let us define
$$B_{\bq}=\left\{A\in \textrm{Mat}_{m\times n}(\mathcal{O}) \colon \underset{\bp\in Z^m}{\inf}\|A\bq-\bp\|_\alpha<\frac{q^R}{\|\bq\|_\beta}\right\}.$$
Modifying the proof of \cite{dkl05} Proposition 3, for $\|\bq\|_\beta=q^k$ we have $$\lambda^{mn}(B_\bq)=q^{\lceil(R-k)a_1\rceil+\cdots+\lceil(R-k)a_m\rceil-m}.$$
Moreover, by Proposition 4 of \emph{loc. cit.} it holds that $$\lambda^{mn}(B_\bq\cap B_{\bq'})=\lambda^{mn}(B_{\bq})\lambda^{mn}(B_{\bq'})$$ for linearly independent vectors $\bq,\bq'\in \bZ^n$. Let $d(\bq)$ be the number of common divisors in $\bZ$ of the coordinates of $\bq $ and let $$\tau_\bq=q^{\lceil(R-k)a_1\rceil+\cdots+\lceil(R-k)a_m\rceil-m}d(\bq),$$ for $\|\bq\|_\beta=q^k$. Since we only get contributions from the elements corresponding to the pairs of parallel vectors $\bq$ and $\bq'$, we have
$$\int\left(\sum_{q^s<\|\bq\|_\beta\le q^t}\chi_{B_\bq}(A)-\sum_{q^s<\|\bq\|_\beta\le q^t}\lambda^{mn}(B_\bq)\right)^2 dA\ll \sum_{q^s<\|\bq\|_\beta\le q^t}\tau_\bq$$ for $s<t$.
The Lemma 10 in (\cite{sp79}) says that given a measure space $(\Xi,\omega)$, if a sequence of nonnegative $\omega$-measurable function $\{f_k\}$ and two sequences of nonnegative real numbers $\{a_k\}$ and $\{b_k\}$ satisfy $0\le a_k\le b_k\le 1$ and 
$$\int_\Xi \left(\sum_{i<k\le j}f_k(x)-\sum_{i<k\le j}a_k\right)^2d\omega\le C\sum_{m<k\le n}b_k$$
for every pair of integers $(i,j)$ with $i<j$, then for $\omega$-almost every $x$
$$\sum_{1\le k\le n}f_k(x)=\sum_{1\le k\le n}a_k+O(B^{\frac{1}{2}}(n)\log^{\frac{3}{2}+\epsilon}B(n))$$
where $B(n)=\sum_{1\le k\le n}b_k$.

It follows that we have for $\lambda^{mn}$-almost every $A$,
\begin{align*}N_R(T,A)&=\sum_{1\le\|\bq\|_\beta\le q^T}\chi_{B_\bq}(A) \\&=\sum_{1\le\|\bq\|_\beta\le q^T}\lambda^{mn}(B_\bq)+O(S(T)^{1/2}\log^{3/2+\epsilon}S(T))
\end{align*}
for $S(T)=\sum_{1\le\|\bq\|_\beta\le q^T}\tau_\bq$. We immediately obtain 
\begin{align*}
&\sum_{1\le\|\bq\|_\beta\le q^T}\lambda^{mn}(B_\bq)=\sum_{k\in\mathbb{Q},0\le k<T}\#\{\|\bq\|_\beta=q^k\}q^{\lceil(R-k)a_1\rceil+\cdots+\lceil(R-k)a_m\rceil-m}\\
=&\sum_{k\in\mathbb{Q},0\le k\le T}(q^{\lfloor ka_{m+1}\rfloor+\cdots+\lfloor ka_{m+n}\rfloor}-q^{\lceil ka_{m+1}\rceil+\cdots+\lceil ka_{m+n}\rceil-n})q^{\lceil(R-k)a_1\rceil+\cdots+\lceil(R-k)a_m\rceil-m}\\
=&\lambda^{m+n}(E_{T,R}).
\end{align*}
It remains to show that $S(T)=O(\Psi_R(T)\log \Psi_R(T))$. In fact, using the Dirichlet series of the number of monic divisors (\cite{rosen}, page 17), we have

\begin{align*}
\displaybreak[0] S(T)&=\sum_{1\le\|\bq\|_\beta\le q^T}\chi_{B_\bq}\sum_{d|(q_1,\cdots,q_n)}1\\
\displaybreak[0] &\ll\sum_{k\in\mathbb{Q},0\le k\le T}\sum_{\l=0}^{\lfloor k\rfloor}\underset{|GCD(q_1,\cdots,q_n)|=|L|=q^l}{\sum_{\|\bq\|_\beta=q^k}}\lambda^{mn}(B_\bq)\sum_{d|L}1\\
\displaybreak[0] &\ll\sum_{k\in\mathbb{Q},0\le k\le T}q^{\lceil(R-k)a_1\rceil+\cdots+\lceil(R-k)a_m\rceil-m}\sum_{l=0}^{\lfloor k\rfloor}\sum_{\|\bv\|_\beta=q^{k-l}}(l+1)q^l\\
\displaybreak[0] &\ll\Psi_R(T)\log\Psi_R(T).
\end{align*}
This completes the proof.
\end{proof} 
Switching the roles of $\bx$ and $\by$, we define
$$F_{S,R}=\left\{(\bx,\by)\in\bK^m\times\bK^n\colon\|\bx\|_\alpha<\frac{q^R}{\|\by\|_\beta},1\le\|\bx\|_\alpha\le q^S\right\}.$$
\begin{lem} 
We have $\lambda(E_{S,R})=\lambda(F_{S,R})$.
\end{lem}
\begin{proof} Since $\lambda(E_{S,R})=\frac{S}{R}\lambda(E_{R,R})$ and $\lambda(F_{S,R})=\frac{S}{R}\lambda(F_{R,R})$, it suffices to show that $\lambda(E_{R,R})=\lambda(F_{R,R})$, equivalent to that $\lambda\{(\bx,\by)\in\bK^m\times\bK^n\colon \|\bx\|_\alpha\le q^R,\|\by\|_\beta< 1\}=\lambda\{(\bx,\by)\in\bK^m\times\bK^n\colon \|\bx\|_\alpha<1,\|\by\|_\beta\le q^R\}$. Since both of these sets have measure $q^{R(a_1+\cdots +a_m)-1}$, we get the conclusion.
\end{proof}

%

For $f\colon \bK^d\to\mathbb{R}$, let us define a function $\widehat{f}$ on $X$ by
$$\widehat{f}(\Lambda)=\sum_{\bv\in \Lambda-\{\mathbf{0}\}}f(\bv).$$

\begin{prop}\label{E and F} For $\lambda^{mn}$-almost every $A$, we have
$$\lim_{N\to\infty}\frac{1}{N}\sum_{n=0}^{N-1}\widehat{\chi}_{E_{T,R}}(g_\ba^nu_A\Gamma)=\lambda(E_{T,R})$$
and
$$\lim_{N\to\infty}\frac{1}{N}\sum_{n=0}^{N-1}\widehat{\chi}_{F_{S,R}}(g_\ba^nu_A\Gamma)=\lambda(F_{S,R}).$$
\end{prop}
\begin{proof} It holds that 
\begin{align*}
\#\left[\Lambda\cap(E_{N,R}\backslash E_{T,R})\right]&\le\frac{1}{T}\sum_{n=0}^{N-1}\widehat{\chi}_{E_{T,R}}(g_\ba^n\Lambda)\\
&\le\#\left[\Lambda\cap(E_{N+T,R})\right]
\end{align*}
for all lattices $\Lambda$. Now
\begin{align*}
\frac{T}{N}\#\left[u_A\Gamma\cap(E_{N,R}\backslash E_{T,R})\right]&\le\frac{1}{N}\sum_{n=0}^{N-1}\widehat{\chi}_{E_{T,R}}(g_\ba^nu_A\Gamma)\\
\displaybreak[0] &\le\frac{T}{N}\#[u_A\Gamma\cap E_{N+T,R}]
\end{align*}
together with Theorem \ref{counting} implies the first equation. 
For the second equality, if we let $F=\{(\bx,\by)\in \bK^m\times\bK^n\colon \|\bx\|_\alpha\le q^S, \|\by\|_\beta\le q^R\}$, then
$$\frac{1}{S}\sum_{n=0}^{N-1}\widehat{\chi}_{F_{S,R}}(g_\ba^nu_A\Gamma)\le \#[u_A\Gamma\cap(E_{N+R,R})]+\#[F\cap u_A\Gamma].$$
Therefore, $\lim_{N\to\infty}\frac{1}{N}\sum_{n=0}^{N-1}\widehat{\chi}_{F_{S,R}}(g_\ba^nu_A\Gamma)\le\lambda(E_{S,R})=\lambda(F_{S,R})$. Conversely, there is a full measure subset $M'$ of $M$ for which $$\lim_{N\to\infty}\frac{1}{N}\sum_{n=0}^{N-1}\varphi(g_\ba^nu_A\Gamma)=\int_X\varphi d\mu$$
holds for all $\varphi\in C_c(X)$ and $A\in M'$. For any $\epsilon>0$, there exists $\varphi_1\in C_c(\bK^d)$ such that $\chi_{F_{S,R}}(\bv)\ge\varphi_1(\bv)$ and $\int_{\bK^d}\varphi_1>\lambda(F_{S,R})-\epsilon$. Moreover, there exists $\varphi_2\in C_c(X)$ such that $\varphi_2(\Lambda)\le \widehat{\varphi_1}(\Lambda)\le\widehat{\chi_{F_{S,R}}}(\Lambda)$ and
$$\int_X\varphi_2 d\mu\ge \lambda(F_{S,R})-2\epsilon.$$
Since $\epsilon$ was arbitrary, for almost every $A\in M$ we have
$$\underset{N\to\infty}{\textrm{liminf}}\frac{1}{N}\sum_{n=0}^{N-1}\widehat{\chi_{F_{S,R}}}(g_\ba^nu_A\Gamma)\ge\lambda(F_{S,R})$$
which completes the proof.
\end{proof}


\section{Equidistribution with respect to $C_\alpha(X)$ and counting with directions}

For a generalized topological space $E$ equipped with a measure $\lambda$ on which a unimodular group $G$ acts transitively and preserving $\lambda$, the mean value theorem of Siegel \cite{si45} has been generalized in \cite{mo96}. In our special case when $E=\bK^d$, $G=SL(d,\bK)$ and $\Gamma=SL(d,\bZ)$, we have
$$\int_{\bK^d}f(\bx)d\lambda=\int_{X}\widehat{f}(\Lambda)d\mu.$$

In order to control the rate of growth at infinity of unbounded functions on $X$, let us use the following notation introduced by \cite{emm98}: For a unimodular lattice $\Lambda\in X$ and a subgroup $\Delta\le\Lambda$ with $L=\bK\Delta$, let us denote by $d(\Delta)$ the volume of $L/\Delta$ and $$\alpha(\Lambda)=\max\{d(\Delta)^{-1}\colon \Delta\le\Lambda\}.$$

\begin{lem}\label{comparison} For any $d$ and all sufficiently large $r$, there are constants $c_1$ and $c_2$ such that if $\chi_{B_{q^r}}$ is the characteristic function of the open ball of radius $q^r$ centered at origin in $\bK^d$, then for all $\Lambda\in X$, we have
$$c_1\alpha(\Lambda)\le\widehat{\chi}_{B_{q^r}}(\Lambda)\le c_2\alpha(\Lambda).$$
\end{lem}
\begin{proof} Since $\widehat{\chi}_{B_{q^r}}$ and $\alpha$ is invariant under the left action of the group $SL(d,\mathcal{O})$, 
we may assume that $\Lambda=t^{a_1}\bZ e_1+\cdots +t^{a_d}\bZ e_d$ with $$a_1\le a_2\le\cdots\le a_k\le 0<a_{k+1}\le\cdots\le a_j\le r<a_{j+1}\le\cdots\le a_d$$
and
$$\sum_{i=1}^d a_i=0.$$
In this case, $\widehat{\chi}_{B_{q^r}}=\#(B_{q^r}\cap\Lambda)=q^{jr-(a_1+\cdots+a_j)}$ and $$\alpha(\Lambda)=q^{-a_1-\cdots-a_k}=q^{a_{k+1}+\cdots+a_d}.$$
Therefore, it follows that
$$q^{-dr}\alpha(\Lambda)\le\widehat{\chi}_{B_{q^r}}\le q^{dr}\alpha(\Lambda).$$
\end{proof}

\begin{lem}\label{dominate} For a given $r>0$, there exist $T,R>0$ such that
$$\widehat{\chi}_{B_{q^r}}\le\widehat{\chi}_{E_{T,R}}+\widehat{\chi}_{F_{T,R}}$$
\end{lem}
\begin{proof}
We observe that $\Lambda\cap B_{q^d}\ne\phi$ for all unimodular lattices $\Lambda$ in $\bK^d$. Moreover, for any $\bv\in \Lambda$ we have $\#[B_{q^r}\cap\Lambda]=\#[(B_{q^r}+\bv)\cap\Lambda]$. Meanwhile, there exist $T$ and $R$ such that $B_{q^r}+\bv\subset E_{T,R}\cup F_{T,R}$ and hence $\chi_{B_{q^r}+\bv}\le\chi_{E_{T,R}}+\chi_{F_{T,R}}$. Therefore $\#[B_{q^r}\cap\Lambda]\le\#[E_{T,R}\cap\Lambda]+\#[F_{T,R}\cap\Lambda]$ holds for every unimodular lattice $\Lambda$ in $\bK^d$.
\end{proof}

\begin{lem}\label{ball} Let $B_{q^r}$ be the open ball of radius $q^r>0$ centered at zero in $\bK^d$. Then
we have $$\lim_{N\to\infty}\frac{1}{N}\sum_{n=0}^{N-1}\widehat{{\chi}_{B_{q^r}}}(g_\ba^n u_A\Gamma)=\lambda(B_{q^r})$$ for $\lambda^{mn}$-almost every $A\in \textrm{Mat}_{m\times n}(\bK)$.
\end{lem}
\begin{proof} It follows directly from Corollary 5.4 of \cite{ksw16}, Proposition \ref{E and F} and Lemma \ref{dominate}.
\end{proof}

Lemma \ref{comparison} and \ref{ball} implies the following theorem.
\begin{thm}[$(\{g_\ba^n\}_{n\ge 1}, C_\alpha(X))$-genericity] For $\lambda^{mn}$-almost every $A \in \textrm{Mat}_{m\times n}(\bK)$, we have $$\lim_{N\to\infty}\frac{1}{N}\sum_{n=0}^{N-1} f(g_\ba^n u_A\Gamma)=\int_X f d\mu.$$
\end{thm}
\begin{thm}[Counting with directions] As $T\to\infty$,
$$\#\{u_A\Gamma\cap E_{T,R}(C_1,C_2)\}\sim \lambda(E_{T,R}(C_1,C_2)).$$
\end{thm}
\begin{proof} Since $\#\{E_{T,R}(C_1,C_2)\}=\#\{E_{1,R}(C_1,C_2)\}\cdot T$, we have
\begin{align*}
\#\{u_A\Gamma\cap E_{T,R}(C_1,C_2)\}&\sim\frac{1}{r}\sum_{n=0}^{T-1}\widehat{\chi}_{E_{r,R}(C_1,C_2)}(g_\ba^nu_A\Gamma)\\
&\sim\frac{T}{r}\lambda(E_{r,R}(C_1,C_2))\\
&=\lambda(E_{T,R}(C_1,C_2)).
\end{align*}
This completes the proof.
\end{proof}

\end{document}